\documentclass[a4paper, reqno, 11pt]{amsart}

\usepackage[T1]{fontenc}
\usepackage{ae,aecompl}

\usepackage{amssymb,amscd}
\usepackage{url, calc}
\usepackage{bbm, color}
\usepackage{amsthm,amsfonts,amsmath,amsxtra}
\usepackage{graphicx}

\numberwithin{equation}{section}

\newtheoremstyle{bold}
{.5\baselineskip}{.5\baselineskip}{\itshape}{}{\bfseries}{.}{.5em}{}
\newtheoremstyle{shy}
{.5\baselineskip}{.5\baselineskip}{}{}{\bfseries}{.}{.5em}{}

\makeatletter
\def\@captionfont{\small}
\makeatother

\renewcommand{\ge}{\geqslant}
\renewcommand{\geq}{\geqslant}
\renewcommand{\le}{\leqslant}
\renewcommand{\leq}{\leqslant}

\theoremstyle{bold}
\newtheorem{theorem}{Theorem}[section]
\newtheorem{proposition}[theorem]{Proposition}
\newtheorem{lemma}[theorem]{Lemma}
\newtheorem{corollary}[theorem]{Corollary}

\theoremstyle{shy}
\newtheorem{definition}[theorem]{Definition}
\newtheorem{remark}[theorem]{Remark}
\newtheorem{example}[theorem]{Example}

\DeclareMathOperator{\supp}{supp}

\newcommand{\cB}{\mathcal{B}}
\newcommand{\cC}{\mathcal{C}}
\newcommand{\cD}{\mathcal{D}}

\newcommand{\cM}{\mathcal{M}}

\newcommand{\cU}{\ts\ts\mathcal{U}}
\newcommand{\cV}{\mathcal{V}}
\newcommand{\cW}{\mathcal{W}}

\newcommand{\one}{\mathbbm{1}}

\newcommand{\EE}{\mathbb{E}}

\newcommand{\NN}{\mathbb{N}}

\newcommand{\PP}{\mathbb{P}\ts}

\newcommand{\RR}{\mathbb{R}}

\newcommand{\TT}{\mathbb{T}}

\newcommand{\ZZ}{\mathbb{Z}\ts}

\newcommand{\ts}{\hspace{0.5pt}}

\title[The symbiotic branching model: duality and interfaces]{The symbiotic branching model:\\ duality and interfaces}

\author{Jochen~Blath}
\thanks{This work was supported by the DFG Priority Programme 1590 `Probabilistic Structures in Evolution'.}

\address{TU Berlin, Fakult\"at II, Institut f\"ur Mathematik, MA 7-3,
Stra\ss{}e des 17.\ Juni 136,
10623 Berlin,
Germany. {\tt blath@math.tu-berlin.de}.}

\author{Marcel~Ortgiese}

\address{Department of Mathematical Sciences, University of Bath, Claverton Down, Bath, BA2 7AY,
United Kingdom. {\tt m.ortgiese@bath.ac.uk}.}

\begin{document}

\begin{abstract}
The symbiotic branching model describes the dynamics of a spatial two-type population, 
where locally particles branch at a rate given by the frequency of the other type combined with nearest-neighbour migration. 
This model generalizes various classic models in population dynamics, such as the 
stepping stone model and the mutually catalytic branching model. 
We are particularly interested in understanding the region of coexistence, i.e.\  the interface between the two types. In this chapter, we give an overview over our results that describe the dynamics of these interfaces 
at large scales.
One of the reasons that this system is tractable is that it exhibits a rich duality theory. So at the same time, we 
take 
the opportunity to provide an introduction to the strength of duality methods in the context of spatial population models.
\end{abstract}

\maketitle

%\frontmatter
%\title{Probabilistic Structures in Evolution}
%
%\author{Ellen Baake \and Anton Wakolbinger \and (Editors)}

%\begin{center}
%{\Large \bf The symbiotic branching model: duality and interfaces
%}\\[5mm]
%
%
%\vspace{0.7cm}
%\textsc{Jochen~Blath and Marcel~Ortgiese\footnote{Department of Mathematical Sciences, University of Bath, Claverton Down, Bath, BA2 7AY,
%United Kingdom, {\tt m.ortgiese@bath.ac.uk}.}}

%\end{center}

% !TEX root = spp1590.tex
%\mychapter{The symbiotic branching model: duality and interfaces}{Jochen~Blath and Marcel~Ortgiese}

% alter Titel: Spatial population dynamics and duality

\newcommand{\MOJBssup}[1] {{\scriptscriptstyle{({#1}})}}
\newcommand{\MOJBsse}[1] {{\scriptscriptstyle{[{#1}}]}} %%%% for small superscript mit eckigen Klammern

%A figure has the following format:0

%\begin{figure}[t]
%\begin{center}
    %  \includegraphics[width=0.5\textwidth]{file}
%\end{center}
%\caption{\label{MOJB-MOJB-figlabel} Figure caption}
%\end{figure}

%All labels must start with the project number, to make
%them distinguishable.
%

\section{Introduction}

Over recent years spatial stochastic models have become increasingly  important in population dynamics. Of particular interest are the spatial patterns that emerge through the interaction of different types via competition, spatial colonization, predation and (symbiotic) branching.
The classic model in this field is the stepping stone model of Kimura~\cite{MOJB-kimura1953stepping}. 
More recent developments include~\cite{MOJB-Zaehle2005,MOJB-bolker1999spatial,MOJB-BEM07, MOJB-Barton2010}, see also the contributions by  Birkner/Gantert and Greven/den Hollander in this volume.

A particularly useful technique in this context is duality.\index{duality} This technique allows to relate two (typically Markov) processes in such a way that information e.g.\ about the long-term behaviour of one process can be translated to the other one. The most basic form of duality can be described as follows: we say that two stochastic processes $(X_t)_{t \geq0}$ and $(Y_t)_{t \geq 0}$ with state spaces $E_1$ and $E_2$ are dual with respect to a (measurable) duality function $F : E_1 \times E_2 \rightarrow \RR$ if for any $x \in E_1, y \in E_2$,
\begin{equation}\label{MOJB-eq:duality}
 \EE_{x} [ F(X_t, y) ] = \EE_{y} [ F(x,Y_t) ] . \end{equation}
The particular case  when $F(x,y) = x^y$ is known as moment duality and holds e.g.\ for a Wright-Fisher diffusion with dual given by the block-counting process of the Kingman coalescent. There are also other variations of duality such as pathwise duality where both original process and dual process can be constructed on the same probability space. Pathwise duality often arises when tracing back genealogies in population dynamics, see also the contributions by Birkner/Blath, Blath/Kurt and Kersting/Wakolbinger in this volume.

To date, there is no general theory that characterizes all possible duals or even just guarantees existence. However, if a dual process exists, exploiting this duality can often be a powerful way of analysing a model.
See~\cite{MOJB-JK14} for a  survey on duality, \cite{MOJB-Sturm2018} for results on pathwise duality in a general setting, but also~\cite{MOJB-SSV18} for a survey of recent developments regarding a systematic approach to duality based on~\cite{MOJB-GKRV09, MOJB-Carinci2015}.

In this chapter we will mostly focus on a class of processes known as the \emph{symbiotic branching model}\index{symbiotic branching model} introduced in~\cite{MOJB-EF04}. These models describe the dynamics of a spatial two-type population that interacts through mutually modifying their respective branching rates. 

However, before we will look at the symbiotic branching model, we will set the scene in Section~\ref{MOJB-sec:discrete_voter} by considering the discrete-space voter model, one of the classic spatial population models, which also has a close connection to the symbiotic branching model. We will show how a basic duality arises in this context and indicate how it can be used to determine the long-term behaviour as well as to describe the interfaces between different types. 
In Section~\ref{MOJB-sec:SBM}, we then introduce the symbiotic branching model and in particular describe our results regarding the interfaces between different types that we characterize via a scaling limit. Note that the symbiotic branching model is particularly interesting as it exhibits several natural yet different kinds of dualities: a {\em self-duality} that we describe in Section~\ref{MOJB-sec:self_duality}, and a {\em moment duality} considered in Section~\ref{MOJB-sec:moment_duality}. In  Section~\ref{MOJB-sec:interface_duality}, we look at how we can use the moment duality in a special case to gain insight into the scaling limit of the system. It turns out that this scaling limit is closely related to a continuous-space version of the voter model and that its interfaces here are described by annihilating Brownian motions, giving rise to an {\em interface duality}. We exploit this connection between the spatial population model and its interface in Section~\ref{MOJB-sec:entrance_laws} to 
characterize the entrance laws of annihilating Brownian motions. 
Finally, in Section~\ref{MOJB-sec:outlook} we briefly discuss open problems in this area.

\section{The discrete-space voter model}\label{MOJB-sec:discrete_voter}

Our first (well-known) example for duality in a spatial population model arises in the context of the classic voter model.
Informally, the voter model represents a population indexed by $x \in \ZZ^d$, where each individual has an opinion $0$ or $1$. At rate $1$ each individual uniformly picks a neighbour and then copies  the opinion of the chosen neighbour.
An alternative interpretation  is that of a biological population of two different types such that at rate $1$ an individual dies and is replaced by the type of a parent uniformly chosen from the neighbours. If the underlying graph is the complete graph, the voter model is a version of the  Moran model, see e.g.~\cite[Sec.\ 1.5]{MOJB-Du08}. We will also see that variations of the voter model arise as a limit when looking at more complicated population models indexed by $\ZZ^d$.
The classic reference for the voter model is~\cite{MOJB-Liggett85}, see~\cite{MOJB-Swart17} for a more recent exposition.  A  formal definition of the system is the following.

\begin{definition} The \emph{voter model}\index{voter model}
 is a Markov process $(\eta_t)_{t\ge0}$ taking values in $\{0,1\}^{\ZZ^d}$ 
such that if the current state is $\eta = (\eta(x))_{x \in \ZZ^d}$, then	
 \begin{equation}\label{MOJB-eq:voter-discrete}
 \eta(x)\text{ flips to }1-\eta(x) \text{ at rate }\frac{1}{2d}\sum_{y:|y-x|=1}\one_{\{\eta(y)\ne \eta(x)\}},\qquad x\in\ZZ^d.
 \end{equation}
\end{definition}

The voter model is famously characterized by the following  duality: 
For all $\eta\in\{0,1\}^{\ZZ^d}$ and finite subsets $A\subset\ZZ^d$, we have
\begin{equation}\label{MOJB-eq:voter_duality}
\EE_{\eta}\Big[\prod_{x\in A}\eta_t(x)\Big]=\EE_{A}\Big[\prod_{x\in {\mathbf Y}_t} \eta(x)\Big],\qquad t\ge0,
\end{equation}
where $({\mathbf Y}_t)_{t\ge0}$ denotes a (set-valued) system of (instantaneously) coalescing nearest-neighbor random walks starting from $A$.\index{coalescing random walks}

One particularly nice way to analyse the voter model,  which also gives the duality~\eqref{MOJB-eq:voter_duality}, is via a \emph{graphical construction}\index{graphical construction} due to~\cite{MOJB-Harris78}: 
We write $x \sim y$ if $x$ and $y$ are neigbhours in $\ZZ^d$ and for $x \sim y$, we denote by $(x,y)$ the directed edge from $x$ to $y$. Then, let $( N_{(x,y)}, x,y \in \ZZ^d, x\sim y) $ be a collection of independent 
 Poisson point processes on $\RR^+$ with rate $\frac{1}{2d}$ each. At an event 
of $N_{(x,y)}$ at time $s$ we draw a directed edge from $(s,x)$ to $(s,y)$, 
so that together with the lines $\ZZ^d \times [0,\infty)$ we obtain a directed graph as in Figure~\ref{MOJB-graphical_construction}.

\begin{figure}[t]
\begin{center}
   \includegraphics[width=0.5\textwidth]{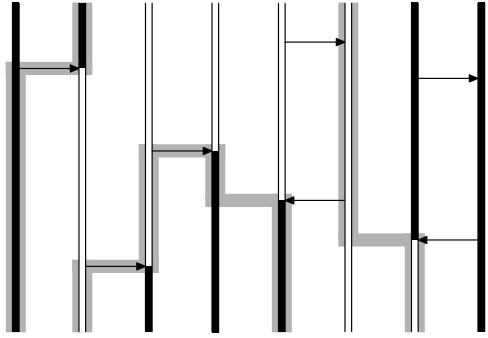}
\end{center}
\caption{\label{MOJB-graphical_construction} The graphical construction of the voter model. The two different types are indicated in black and white and the interface between the two types in grey. }
\end{figure}

We can define the voter model started in an initial condition $\eta \in \{0,1\}^{\ZZ^d}$ as follows: 
the initial opinions are propagated by letting them flow upwards in the graphical 	construction and if they encounter  an arrow by letting them flow along the direction of the arrow (and replacing the opinion at that site if it is different).

For each site $(x,t) \in \ZZ^d \times [0,\infty)$ and $s \in [0,t]$, we
set  $\xi_s^{t,x} = y \in \ZZ^d$ if $y \in \ZZ^d$ is the unique point at time $t-s$
that is reached by starting at $(t,x)$ and following vertical lines downwards and when encountering the tip of an arrow following the arrow horizontally in reverse direction.
An equivalent way to describe the above flow construction is to set 
$\eta_t(x) := \eta( \xi_t^{t,x})$. 

We  note that by the Poisson construction  $\xi^{t,x} = (\xi^{t,x}_s)_{s \in [0,t]}$ has the law of a simple random walk (where jumps occur at rate $1$). Moreover, if we consider the system $\{\xi^{t,x}, x \in A\}$ for a finite set set $A \subset \ZZ^d$, then this collection has the same law as a system of coalescing random walks: each particle moves as an independent random walk until two particles meet. After meeting, the two particles involved in the collision follow  the  same random walk trajectory. For more details see~\cite[Sec.\ III.6]{MOJB-Liggett85}.

From this construction, we have immediately that for $x_1,\ldots, x_n \in \ZZ^d$,
\[ \EE_\eta \Big[ \prod_{i=1}^n \eta_t(x_i)\Big] =\EE\Big[ \prod_{i=1}^n \eta (\xi^{t,x_i}_t)  \Big]
= \EE_{x_1,\ldots,x_n}\Big[ \prod_{y \in {\mathbf Y}_t} \eta(y) \Big], \]
where ${\mathbf Y}_t = \{ Y_0^{t,x_i}, i=1,\ldots,n\}$ is the (set-valued) system of coalescing random walks started in ${\mathbf Y}_0 = \{x_1,\ldots,x_n\}$. Therefore, we have shown~\eqref{MOJB-eq:voter_duality}.

\begin{remark} As mentioned before  the voter model on the complete graph with $n$ vertices is  a variant of the Moran model. See the contribution of Baake/Baake in this volume for graphical constructions with extensions to more general models in this context.
\end{remark}

An immediate consequence of the duality with coalescing random walks is that the system 
in lower dimensions $d=1,2$ experiences clustering.

\begin{proposition} Let $(\eta_t)_{t \geq 0}$ be the the voter model started in $\eta \in \{0,1\}^{\ZZ^d}$ and assume $d \in \{1,2\}$. Then, for any $x,y \in \ZZ^d$,
\[ \PP ( \eta_t(x) = \eta_t(y)  ) \rightarrow 1, \quad \mbox{as } t \rightarrow \infty . \]
\end{proposition}

\begin{proof} 
Note that by the graphical construction
\[ \PP (\eta_t(x) = \eta_t(y) ) = \PP ( \eta(\xi^{t,x}_t) = \eta(\xi^{t,y}_t) )
\geq \PP( \xi^{t,x}_t = \xi^{t,y}_t) = \PP_{x,y}( \tau \leq t), \]
where $\tau$ is the first meeting time of two independent random walks started in $x$ and $y$. 
Since the difference of two random walks is again a random walk which is recurrent in $d=1,2$, 
the latter probability tends to $1$ as $t \rightarrow \infty$. 
\end{proof}

In particular, any invariant measure is concentrated on configurations consisting of all $0$s or all $1$s.
Similarly, one can show that in dimensions $d \geq 3$, due to the transience of the random walk, the invariant measures are not constant. See e.g.\ \cite[Corollary~V.1.13]{MOJB-Liggett85}.

We will now concentrate on the case $d =1$.
A question that we will come back to frequently is whether we can describe the dynamics of the `interfaces' between the two different types. More  formally, consider the interface of a configuration $\eta \in \{0,1\}^{\ZZ}$ as
\[\mathcal{I}(\eta) :=\{x\in\ZZ\,|\, \eta(x)\ne \eta(x+1)\} .\]
Then, we can explicitly describe the law of this process as first observed in~\cite{MOJB-S78}.

\begin{proposition}\label{MOJB-prop:discrete_interface}
Let $\eta \in \{0,1\}^{\ZZ}$. The interface of the voter model $\mathcal{I}(\eta_t)$ with $\eta_0 = \eta$ follows a system of annihilating random walks  started in $\mathcal{I}(\eta)$.
\end{proposition}

Recall that a system of (instantaneously) annihilating random walks\index{annihilating random walks} is a system of  random walks on $\ZZ$ that move independently  until the first collision time of a pair of particles, 
at which point the two particles involved annihilate each other.

\begin{proof} The statement can either be checked by calculating generators, see~\cite{MOJB-S78} or it follows  from the graphical construction, see also Figure~\ref{MOJB-graphical_construction}:
Note that if an interface particle is at site $x$ and encounters an arrow from $x$ to $x+1$
then it jumps to the right. Conversely, if it is at $x$ and encounters an arrow from $x+1$ to $x$, then it jumps to the left.
Since arrows appear at rate $1/2d$ each particles performs a simple random walk and different particles are independent since they use a disjoint set of arrows. 
Finally, if a particle jumps on top of another, then the type to the right of the left particle 
dies out in the voter model and so the interface particles annihilate.
\end{proof}

\index{graphical construction}
\index{duality!interface duality}

This relation between the annihilating random walks and the voter model leads to the following `interface duality':

\begin{corollary} For any $x, y\in \ZZ$ with $x<y$ and denoting by ${\mathbf X} = ({\mathbf X})_{t \geq 0}$ a system of annihilating random walks, we have for any $\eta \in \{0,1\}^{\ZZ}$ and for any $t \geq 0$,
\[ \PP_{\mathcal{I}(\eta)} ( | {\mathbf X}_t \cap [x,y-1] | \mbox{ even} ) = \PP_\eta ( \eta_t(x) = \eta_t(y) ) . \]
\end{corollary}

\begin{proof} This follows from Proposition~\ref{MOJB-prop:discrete_interface} together with the observation that $\eta_t(x) = \eta_t(y)$ iff $\mathcal{I}(\eta_t) \cap [x,y-1]$ is even. 
\end{proof}

In fact this relationship also means that given an initial condition $\eta \in \{0,1\}^{\ZZ}$, one can construct a voter model by first sampling 
a system of annihilating random walks started in $\mathcal{I}(\eta)$ and then uniquely colouring the remaining sites so that the annihilating walks correspond to the interfaces.

A similar interface duality is known for a one-dimensional voter model with swapping, where the  interfaces follow a	symmetric double-branching annihilating random walk, see~\cite{MOJB-S90, MOJB-BK10, MOJB-SS08}. We will come back to this duality in a continuous-space setting, see Sections~\ref{MOJB-sec:moment_duality} and~\ref{MOJB-sec:entrance_laws} below.

\section{The symbiotic branching model}\label{MOJB-sec:SBM}

Our main object of study will be the \emph{symbiotic branching model}\index{symbiotic branching model} introduced by Etheridge and Fleischmann in \cite{MOJB-EF04}. 
The model describes the dynamics of a spatial population consisting of two types. 
In the corresponding infinitesimal particle model, 
locally the population of each type follows a critical branching process, where
the branching rate is given 
by $\gamma$ times the frequency of particles of the other type, where $\gamma  >0$ is a parameter of the model. 
Moreover, each particle migrates according to an independent Brownian motion. Finally, the branching mechanisms are correlated with a correlation parameter denoted by $\varrho \in[-1,1]$. For a more precise description of the particle system, see~\cite{MOJB-EF04}.

In one spatial dimension and in continuous space, the model is described by two interacting stochastic partial differential equations (SPDEs). Here, $u_t(x)$ and $v_t(x)$ describe the densities of each 
type at time $t \geq 0$ and site $x \in \RR$.
The evolution of these (non-negative) densities is given by
\begin{equation}\begin{aligned}
\label{MOJB-eqn:spde}
  			\frac{\partial }{\partial t}
u_t(x) & = \frac{\Delta}{2} u_t(x) + 
                     	\sqrt{ \gamma u_t(x) v_t(x)} \, \dot{W}^\MOJBssup{1}_t(x),\\[0.3cm]
 			\frac{\partial }{\partial t}v_t(x) 
 & = \frac{\Delta}{2} v_t(x) + 
                     	\sqrt{ \gamma u_t(x) v_t(x)} \, \dot{W}^\MOJBssup{2}_t(x),
\end{aligned}\end{equation}
with suitable nonnegative initial conditions $u_0(x) = u(x) \ge 0$ and $v_0(x) = v(x) \ge 0$, $x \in \RR$.  
Here, $\gamma > 0$ is the branching rate, $\Delta$ is the Laplacian and $ (\dot{W}^\MOJBssup{1},
\dot{W}^\MOJBssup{2})$ is a pair of correlated standard Gaussian white noises on $\mathbb{R}_+ \times \mathbb{R}$ with correlation 
parameter $\varrho
\in [-1,1]$. 
We refer to these SPDEs as  $	{\mathrm{cSBM}(\varrho,\gamma)}_{u, v}$.
Existence and uniqueness for these equations are covered in~\cite{MOJB-EF04} (where uniqueness in general is still open for $\varrho =1$).
There is also a discrete space version of the model (e.g.\ indexed by $\ZZ^d$), but we will focus on the spatial continuum.

A main motivation for this model stems from the fact that it generalizes several well-known examples of spatial population models: For $\varrho = -1$ and for initial conditions $u \equiv 1 - v$ one recovers a continuous-space version of the stepping stone model\index{stepping stone model} of Kimura, see also~\cite{MOJB-T95}.	For $\varrho =0$, the model is known as the {\em mutually catalytic branching model}\index{mutually catalytic branching model} due to Dawson and Perkins \cite{MOJB-DP98}. For $\varrho =1$ and if $u \equiv v$, then the system reduces to the \emph{parabolic Anderson model}\index{parabolic Anderson model}, compare the contribution of K{\"o}nig in this volume.  In this case, uniqueness of the system is covered by standard SPDE techniques, see e.g.\ \cite{MOJB-MuellerSupport91}.

In order to investigate the dynamics of the model, one has to understand the 
balance between the critical local branching mechanism, which pushes one type towards extinction, and the Laplacian, which smoothes out solutions and in particular pushes mass back into regions where one type has died out. A particularly interesting consequence of this competition of forces is the observation of~\cite{MOJB-EF04} that 
for any $\varrho \in [-1,1]$ if we start with initial conditions where both types are initially separated, as for example
the complementary Heaviside conditions, i.e.\
\begin{equation}\label{MOJB-eq:initial_conditions} 
u =\one_{(-\infty,0]}  \quad \mbox{and} \quad v = \one_{[0,\infty)}, \end{equation}
then the region where both types coexist remains finite, despite the efforts of the Laplacian to spread mass everywhere instantaneously. More formally, 
define the \emph{region of coexistence} or the \emph{interface}\index{interface} at time $t$  as
\[ 			\mathcal{I}_t = \mathcal{I}(u_t,v_t) = \supp(u_t) \cap \supp(v_t). \]
Then,~\cite{MOJB-EF04} show  that $\mathcal{I}_t$ is a compact set and the width of the interface grows at most linearly in $t$.

\index{coexistence}
\index{region of coexistence}

One of our main goals is to understand the evolution of the interface in more detail. 
In the case $\varrho =-1$, for the stepping stone model with Heaviside initial conditions,  a result by Tribe~\cite{MOJB-T95} shows that after diffusive rescaling the interface shrinks to a single point that moves like a Brownian motion.

One of our earlier works, \cite[Thm.\ 2.11]{MOJB-BDE11} showed that for all $\varrho$ close to $-1$
 there is a constant $C > 0$ such that almost surely, 
for all $t$ large enough, the interface is contained in the set  $[- C\sqrt{t\log(t)}, C \sqrt{t \log(t)}]$.
This shows sub-linear speed for the interface and is consistent with the conjecture that
the diffusive behaviour might also be correct  
for other $\varrho > -1$. This conjecture is also 
 supported by the following scaling property: 
\cite[Lemma 8]{MOJB-EF04} shows that  for any $\gamma > 0, K > 0$,
if $(u_t,v_t)_{t \geq 0}$ is solution of  cSBM$(\varrho,\gamma)_{u,v}$, then 
\begin{equation}\label{MOJB-eq:rescaling} (u_{K^2t}(Kx), v_{K^2 t}(Kx) )_{ x \in \RR, t \geq 0} \quad\mbox{ solves } \quad
{\rm cSBM}(\varrho , K \gamma)_{u^\MOJBssup{K}, v^\MOJBssup{K}} , \end{equation}
where $u^\MOJBssup{K}(x) = u(Kx)$ and $v^\MOJBssup{K} (x) = v(Kx)$ for all $x \in \RR$ are
suitably rescaled initial conditions.
In particular, if the initial conditions are invariant under the rescaling (as e.g.\ the ones  in~\eqref{MOJB-eq:initial_conditions}), then a diffusive space-time rescaling is in law equivalent to rescaling the branching parameter.
In particular, in the following we will be discussing a scaling limit as $\gamma \rightarrow \infty$, 
which also allows us to consider more general initial conditions.

Our first main result shows that at least for negative $\varrho$, the diffusive rescaling indeed 
captures the non-trivial behaviour of the interface.
To formulate the convergence, we move from  densities to measure-valued processes by defining
\begin{equation}\label{MOJB-definition:mu_nu} \mu_t^\MOJBsse{\gamma} (dx)  := u^\MOJBsse{\gamma}_t(x)\, dx ,\quad  \quad \nu_t^\MOJBsse{\gamma} (dx)
:= v^\MOJBsse{\gamma}_t(x) \, dx  , 
\end{equation}
where we now write $(u^\MOJBsse{\gamma}_t, v^\MOJBsse{\gamma}_t)_{t \geq 0}$ for the solution of cSBM$(\varrho,\gamma)$ to emphasize the dependence on $\gamma$.
Also, we denote by 
$\cM_{\rm tem}$ the space of tempered measures on $\RR$, and by
$\cM_{\rm rap}$ the space of rapidly decreasing measures. Informally, 
a measure $\mu$ is in $\cM_{\rm tem}$ (resp.\ $\cM_{\rm rap}$) if $\mu(f) := \langle \mu, f \rangle := \int_\RR f(x) \, \mu(dx)$ is finite for every non-negative function $f$ that is decreasing exponentially fast (resp.\ for any $f$ that is growing slower than exponentially).
Similarly, 
$\cB^+_{\rm tem}$ (resp.\ $\cB^+_{\rm rap}$) denotes the space of nonnegative, tempered (resp.\ rapidly decreasing)
measurable functions, i.e.\ that grow slower than any exponentially growing function (resp.\ decaying faster than any exponentially decay function). For  formal definitions, see~\cite[Appendix~1]{MOJB-BHO15}.

\begin{theorem}[{\cite[Theorem 1.5]{MOJB-BHO15}, \cite[Theorem 2.2]{MOJB-HO15}, \cite[Theorem 2.8]{MOJB-HOV15}}]\label{MOJB-theorem:scaling}
Let  $\varrho\in[-1,0)$. 
If $\varrho \in (-1,0)$ suppose the initial conditions satisfy
$(u,v)\in(\cB_{\rm tem}^+)^2$ 
or $(u,v)\in(\cB_{\rm rap}^+)^2$ and if $\varrho =-1$ suppose the initial conditions are bounded.
Then as $\gamma \rightarrow \infty$, the measure-valued process $(\mu_t^\MOJBsse{\gamma},\nu_t^\MOJBsse{\gamma})_{t\ge0}$ defined by~\eqref{MOJB-definition:mu_nu}
converges in law with respect to the Meyer-Zheng `pseudo-path' topology to 
a measure-valued process $(\mu_t,\nu_t)_{t \geq 0}$. 
Moreover, for any $t >0$, almost surely the limiting measures $\mu_t$ and $\nu_t$
are absolutely continuous with respect to the Lebesgue measure and their densities $u_t, v_t$
satisfy the following separation-of-types property:\\[-5mm]
\[ u_t(x) v_t(x) = 0 \quad \mbox{for almost all } x \in \RR. \]
\end{theorem}

\index{separation of types}

In the following we will refer to the limit  $(\mu_t,\nu_t)_{t \geq 0}$ (or its density $(u_t,v_t)_{t \geq0}$) as the \emph{continuous-space
infinite rate symbiotic branching model} $\mathrm{cSBM}(\varrho,\infty)_{u,v}$.

The theorem is proved by showing tightness and uniqueness of limit points using two different types of duality that are known 
for the symbiotic branching model.
 For tightness, we make use of the duality to Brownian motions with dynamically changing colours (see  Section~\ref{MOJB-sec:moment_duality}) and for uniqueness, we use the self-duality first applied in this context by Mytnik~\cite{MOJB-Mytnik98}, see Section~\ref{MOJB-sec:self_duality}.

\begin{remark}
\begin{itemize}
 \item[(a)] The Meyer-Zheng `pseudo-path' topology is a fairly weak topology on 
 the space of c\`adl\`ag measure-valued processes and essentially requires convergence
 at all times apart from those in a Lebesgue-null set. For a formal definition we refer to~\cite[Appendix A.1]{MOJB-BHO15}.
 Under the more restrictive assumption that  $(u,v)$ is the complementary Heaviside initial condition in~\eqref{MOJB-eq:initial_conditions}, we show in~\cite[Thm.\ 1.12]{MOJB-BHO15} that for $\varrho \in (-1, -\frac{1}{\sqrt{2}})$ and for $\varrho =-1$ in \cite[Thm.\ 2.8]{MOJB-HOV15}, tightness also holds in the stronger (standard) Skorokhod topology, 
 implying convergence  in $\cC_{[0,\infty)}( \cM_{\rm tem}^2)$.
 \item[(b)] The fact that the limiting measures are absolutely continuous is first  shown
for $\varrho \in (-1, -\frac{1}{\sqrt{2}})$ and complementary Heaviside initial conditions in~\cite{MOJB-BHO15} and the statement is extended to 
all $\varrho \in (-1,0)$ in~\cite{MOJB-HO15} and $\varrho =-1$ in~\cite{MOJB-HOV15}.
 \item[(c)]  The reason for the different assumptions on the initial conditions depending on $\varrho$ is mainly technical: for $\varrho =-1$ the result is proved using a moment duality, which will be discussed in Section~\ref{MOJB-sec:moment_duality} and for $\varrho \in (-1,0)$ we rely on a self-duality, see Section~\ref{MOJB-sec:self_duality}.
 \end{itemize}
\end{remark}

As mentioned before, for $\varrho = -1$ and complementary Heaviside initial conditions, the result of Theorem~\ref{MOJB-theorem:scaling} is proved in Tribe~\cite[Thm.\ 4.2]{MOJB-T95} (in the stronger Skorokhod topology) and 
it is shown that the limiting process in this case is
\begin{equation}\label{MOJB-eq:tribe} (\one_{\{x\le B_t\}}\, dx,\one_{\{x\ge B_t\}}\, dx)_{t \geq 0} , \end{equation}
for $(B_t)_{t \geq 0}$ a standard Brownian motion. See also Theorem~\ref{MOJB-theorem:annihilating-BM} below for an extension of this result to more general initial conditions (but still $\varrho = -1$).
Unfortunately, for $\varrho > -1$, we do not currently have a  description of the limit that is as explicit. However, we can characterize the limit via a martingale problem, see also Section~\ref{MOJB-sec:self_duality} below. In particular, in this way we can exclude that the limit is of the simple form~\eqref{MOJB-eq:tribe}. The difference compared to $\varrho =-1$ is that for $\varrho > -1$, the sum $u_t^\MOJBsse{\gamma} + v_t^\MOJBsse{\gamma}$ is no longer deterministic and the random height fluctuations influence the dynamics of the interface.

Even though we do not have an explicit description of the limiting object, we can say more about the interface process. For any Radon measure $\mu$ denote by 
${\rm supp}(\mu) := \{ x \in \RR \, : \, \mu(B_\varepsilon(x)) > 0\mbox{ for all } \varepsilon > 0\}$,
the measure-theoretic support of $\mu$,
where $B_\varepsilon(x)$ denotes a ball of radius $\varepsilon$. Then, define 
\[ L(\mu) := \inf {\rm supp}(\mu), \quad \mbox{and}\quad R(\mu) := \sup {\rm supp}(\mu) . \]

\begin{theorem}[\cite{MOJB-HO15}]\label{MOJB-theorem:point_interface} Suppose $\varrho \in (-1,0)$. Let $(\mu,\nu)$ be initial conditions 
in $\cM_{\rm tem}^2$ or $\cM_{\rm rap}^2$ which are mutually singular and such that
$R(\mu) \leq L(\nu)$ with $\mu + \nu \neq 0$. Let $(\mu_t, \nu_t)_{t \geq 0}$ be
a solution of cSBM$(\varrho, \infty)_{\mu,\nu}$.
 Then,
almost surely, 
\[ R(\mu_t) \leq L(\nu_t) \quad \mbox{for all } t \geq 0.\]
Moreover,
for all fixed $t > 0$, almost surely, $(\mu_t,\nu_t)$ has a single-point interface in the sense that
\[R(\mu_t) = L(\nu_t).\]
\end{theorem}

\begin{remark} 
Note the difference in the order of the quantifiers for $t$ and $\omega$: For the first statement the  set of exceptional $\omega$'s does not depend on $t$, while in the second statement it does.
 Note also that the condition $R(\mu_t) \leq L(\nu_t)$ means that the support of $\mu_t$ 
is to the left of the support of $\nu_t$, but there might possibly be a gap.
In particular, our theorem does not guarantee the existence of an interface process $I_t := R(\mu_t) = L(\mu_t)$.
In general, it is a non-trivial task to obtain results that are uniform in time, see e.g.\ the discussion in \cite[Sec.~7]{MOJB-Dawsonetal2002a} for the two-dimensional mutually catalytic model.
\end{remark}

Unlike for the first result, Theorem~\ref{MOJB-theorem:scaling}, the proof of Theorem~\ref{MOJB-theorem:point_interface} does not directly rely on the technique of duality. 
Instead, we deduce the fact that the interface is a single point by establishing a connection between the continuum space model and the model defined on $\ZZ$. 

The proof heavily relies on prior work for the discrete model. Inspired by the scaling property~\eqref{MOJB-eq:rescaling}, Klenke and Mytnik~\cite{MOJB-KM10a, MOJB-KM11b, MOJB-KM11c} consider the mutually catalytic branching model, i.e.\ the symbiotic branching model with $\varrho =0$, on a discrete space and show that without a spatial rescaling but taking $\gamma \rightarrow \infty$, the model converges to an infinite-rate limiting process.
In contrast to our result, they have an explicit description of the limit in terms of an interacting system of jump-type SDEs and they also study long-term properties of the system. 
Moreover, \cite{MOJB-KO10} give a Trotter type approximation and \cite{MOJB-DM11a, MOJB-DM12} extend the analysis  for the discrete model to all $\varrho \in (-1,1)$  and obtain comparable results to the $\varrho =0$ case.

We show in~\cite{MOJB-HO15} that if one starts with the infinite-rate symbiotic branching model defined on a lattice and takes a diffusive time and space rescaling, then one ends up with the
continuous-space infinite-rate symbiotic branching model. Our strategy in showing Theorem~\ref{MOJB-theorem:point_interface} is then to first use the explicit description of the limit to show that the discrete model started in complementary Heaviside initial conditions has a single-point interface (interpreted suitably) and then 
to show that this property is preserved  in the space-time limit.

\section{Self-duality in the symbiotic branching model}\label{MOJB-sec:self_duality}

The symbiotic branching model has a very rich duality structure. In this section we explain its 
self-duality, which is particularly useful for showing
 uniqueness of solutions, e.g.\ for the SPDE~\eqref{MOJB-eq:initial_conditions}. This duality is essentially based on an idea of Mytnik~\cite{MOJB-Mytnik98}, who used it to show uniqueness for the mutually catalytic branching model.
The reason that self-duality is needed is that as soon as $\varrho > -1$, the densities $(u_t,v_t)$ can have random heights. Therefore, if one wants to show uniqueness of solutions by showing  that moments converge, one is additionally faced with the highly non-trivial task of controlling the growth of these moments.

Self-duality is a classical duality in the sense of~\eqref{MOJB-eq:duality}, where however the dual follows the same dynamics as the original process, but starts with different initial conditions. We start by defining the corresponding duality function. 
Let $\varrho\in(-1,1)$ and if either $(\mu,\nu,\phi,\psi)\in\cM_{\rm tem}^2\times(\cB_{\rm rap}^+)^2$ or $(\mu,\nu,\phi,\psi)\in\cM_{\rm rap}^2\times(\cB_{\rm tem}^+)^2$, 
define
\begin{equation}\label{MOJB-self-duality product}
 \langle\langle \mu , \nu, \phi, \psi \rangle \rangle_\varrho := -\sqrt{1-\varrho}\,\langle \mu + \nu, \phi + \psi\rangle
+ i \sqrt{1+\varrho}\, \langle \mu - \nu, \phi - \psi \rangle ,
\end{equation}
where $\langle \mu , \phi \rangle:=\int_\RR \phi(x)\, \mu(dx)$.
We then define the \emph{self-duality function $F$} as
\begin{equation}\label{MOJB-self-duality function}
 F(\mu,\nu,\phi,\psi) := \exp \langle \langle \mu, \nu, \phi, \psi \rangle \rangle_\varrho. 
\end{equation}

The duality function $F$ also plays an important role in the limiting martingale problem as the following lemma motivates.

\begin{lemma}\label{MOJB-le:mg_problem} Let $\varrho \in (-1,1)$ and suppose the initial conditions satisfy $(u,v) \in (\cB_{\rm tem}^+)^2$ (resp.\ $\in (\cB_{\rm rap}^+)^2)$. Denote by $(u_t,v_t)_{t\geq 0}$ the solution 
of cSBM$(\varrho,\gamma)_{u,v}$ with $\gamma < \infty$ and set $\mu_t(d x) = u_t (x) \, dx$ and $\nu_t(d x)   =  v_t(x) \, dx$.
Then, 
there exists an increasing c\`adl\`ag $\cM_{\rm tem}$-valued (resp.\ $\cM_{\rm rap}$-valued) process $(\Lambda_t)_{t\ge0}$ with $\Lambda_0=0$ and
such that for all twice continuously differentiable test functions 
 $\phi,\psi\in \cB_{{\rm rap}}^+$ (resp.\ $\phi,\psi\in \cB_{{\rm tem}}^+$) 
the process
\begin{equation}\label{MOJB-eq:martingale_problem}\begin{aligned}
 F(\mu_t, \nu_t,\phi,\psi)  & - F(\mu_0,\nu_0,\phi,\psi) \\ &
 - \frac{1}{2}\int_0^t F(\mu_s, \nu_s,\phi,\psi)\,\langle\langle \mu_s, \nu_s, \Delta\phi, \Delta\psi\rangle\rangle_\varrho \, ds \\
& - 4(1-\varrho^2)\int_{[0,t]\times\RR} F(\mu_s,\nu_s,\phi,\psi)\,\phi(x)\psi(x)\,\Lambda(ds,dx)
\end{aligned}\end{equation}
is a martingale. Here, 
\begin{equation}\label{MOJB-eq:lambda} 
\Lambda (dt, dx) := \gamma \, u_t(x) v_t(x) \, dt \, dx, 
\end{equation}
and $\Lambda_t(dx) = \Lambda([0,t]\times  dx)$.
\end{lemma}

\begin{proof} This lemma can be proved by first writing the solution $(u_t,v_t)_{t \geq 0}$
in the weak formulation 
and then applying It\^o's lemma. See the proof of Proposition~5 in~\cite{MOJB-EF04} for details.
\end{proof}

We say that a process $(\mu_t,\nu_t)_{t \geq0}$  taking values in $\cM_{\rm tem}$, resp.\ $\cM_{\rm rap}$,  is a solution of the martingale problem ${\bf MP}_F(\varrho)$ if there exists  an increasing, c\`adl\`ag process
$(\Lambda_t)_{t \geq 0}$ with $\Lambda_0 = 0$ taking values in $\cM_{\rm tem}$, resp.\ $\cM_{\rm rap}$,
 such that the expression in~\eqref{MOJB-eq:martingale_problem} is  a martingale. 
Here, we interpret $(\Lambda_t)_{t \geq 0}$ as a measure on $[0,\infty) \times \RR$ 
by setting $\Lambda([0,t] \times B) = \Lambda_t(B)$ for any Borel set $B$ and $t \geq 0$.
Solutions of the martingale
problem ${\bf MP}_F(\varrho)$ are not unique, as for any $\gamma$, a solution cSBM$(\varrho,\gamma)$ gives a solution of ${\bf MP}_F(\varrho)$.
However, specifying the correlation via~\eqref{MOJB-eq:lambda} fixes solutions. This can be shown 
via the following self-duality\index{duality!self duality} based on an idea of Mytnik~\cite{MOJB-Mytnik98}.

 \begin{lemma}  Fix $\gamma \in (0,\infty)$. Let  $(\mu_t,\nu_t)_{t \geq 0}$ be a process taking values in $\cM_{\rm tem}^2$  
 with densities $(u_t,v_t)_{t \geq 0}$
 that satisfies
 the martingale problem ${\bf MP}_F(\varrho)$ together with \eqref{MOJB-eq:lambda}.
Then, for any test functions~$\phi,\psi \in \cB_{\rm rap}^+$,
 \begin{equation}\label{MOJB-eq:self_duality} \EE_{u_0,v_0} [ F(\mu_t,\nu_t,\phi,\psi)] = \EE_{\phi,\psi} [F( \tilde \mu_t, \tilde \nu_t, u_0,v_0 ) ], \end{equation}
 where $(\tilde \mu_t,\tilde \nu_t)_{t \geq 0}$ is any solution of the martingale problem ${\bf MP}_F(\varrho)$ with densities $(\tilde u_t,\tilde v_t)_{t \geq0}$ satisfying~\eqref{MOJB-eq:lambda} (with $(\tilde u_t,\tilde v_t)$ replacing  $(u_t,v_t)$) and taking values in $\cM_{\rm rap}^2$ with initial conditions $(\tilde u_0,\tilde v_0) = (\phi, \psi)$. 
 \end{lemma}
 
 Clearly the collection of functions $F(\cdot, \cdot, \phi,\psi)$ for $\phi,\psi$ as above is measure-determining since it is a mixed Laplace-Fourier transformation and so the self-duality uniquely determines solutions of the martingale problem if~\eqref{MOJB-eq:lambda} is also specified.
 
The main idea in~\cite{MOJB-BHO15} is to characterize the $\gamma = \infty$ limit via the martingale problem, but instead of explicitly prescribing the correlation as in~\eqref{MOJB-eq:lambda} we replace this condition by a separation-of-types condition.

\begin{theorem}[{\cite[Thm.\ 1.10]{MOJB-BHO15}}] Given initial conditions $(u,v) \in (\cB_{\rm tem}^+)^2$ (resp.\ $\in (\cB_{\rm rap}^+)^2)$, the limiting process cSBM$(\varrho,\infty)$ in Theorem~\ref{MOJB-theorem:scaling} can be characterized as the unique  solution $(\mu_t,\nu_t)_{t \geq 0}$ of 
${\bf MP}_F(\varrho)$ with $\mu_0(dx) = u(x) \, dx$ and $\nu_0(dy) = v(y) \, dy$ for 
which the increasing, c\`adl\`ag process $(\Lambda_t)_{t \geq 0}$ satisfies
\[ \EE_{\mu,\nu}\big[\Lambda_t(dx)\big]\in\cM_{\rm tem}, \qquad(\text{resp.\ } \EE_{\mu,\nu}\big[\Lambda_t(dx)\big]\in\cM_{\rm rap})\] 
and, 
for all $t>0$ and $x\in\RR$,
\begin{equation}\label{MOJB-eq:sep_types} 
\EE_{u, v} [ S_\varepsilon \mu_t(x) S_\varepsilon \nu_t(x)]  \rightarrow 0  \qquad \mbox{as } \varepsilon \rightarrow 0,
\end{equation}
where $(S_t)_{t \geq 0}$ denotes the heat semigroup.
\end{theorem}

The proof of the uniqueness shows that under the 
the separation-of-types condition~\eqref{MOJB-eq:sep_types} any solution of ${\bf MP}_F(\varrho)$ satisfies a self-duality analogous to~\eqref{MOJB-eq:self_duality}.

A similar martingale problem was also used in earlier work for the infinite-model on a discrete space, see~\cite{MOJB-KM10a, MOJB-KM11b}. The difference is that in this context one can work with functions $F$, where $\phi$ and $\psi$ are such that $\phi \psi = 0$, so that the term involving $\Lambda$ in~\eqref{MOJB-eq:martingale_problem} vanishes. Also, the proof of the self-duality relation becomes easier as one does not have to worry about spatial regularity.

\section{Moment duality in the symbiotic branching model}\label{MOJB-sec:moment_duality}

The finite-rate symbiotic branching model also satisfies a moment duality 
as shown in~\cite{MOJB-EF04}. We will first recall this standard construction and then discuss the extension of the moment duality to the infinite-rate model due to~\cite{MOJB-HOV15}. This construction works in $\RR$ as well as in $\ZZ^d$, but we will concentrate on the continuous-space case here.

In this section we denote by $(u_t,v_t)_{t \geq 0}$  (the density of) a solution of cSBM$(\varrho,\gamma)$ for $\gamma \in (0,\infty]$.
For any functions $u,v : \RR \rightarrow [0,\infty)$ and a colour $c \in \{1,2\}$, we introduce the notation 
\begin{equation}\label{MOJB-eq:colour_prod} (u,v)^\MOJBssup{c}(x)  := \left\{ \begin{array}{ll} u(x) & \mbox{if } c =1, \\ v(x) & \mbox{if } c = 2. \end{array} \right.\end{equation}

Then, with this notation and for $n \in \NN$, ${\mathbf x} = (x_1, x_2, \ldots,x_n) \in \RR^n$ and a colouring $c = (c_1,\ldots,c_n) \in \{1,2\}^n$, the duality gives us an expression for the moment
\[ \EE\Big[ \prod_{i=1}^n (u_t,v_t)^\MOJBssup{c_i}(x_i) \Big] . \]
The dual process is given by  $n$ independent  Brownian motions $X = \linebreak (X^1_t, \ldots \ ,  X^n_t)_{t \geq 0}$ started in ${\mathbf x}$.
Moreover, to each Brownian motion we associate a dynamically changing colour process $(C_t)_{t \geq 0}$ with $C_t = (C_t^1, \ldots, C_t^n) \in \{1,2\}^n$, where $C_t^i$ describes the colour  of the $i$th motion and is such that $C_0^i = c_i$ for each $i =1,\ldots,n$. The dynamics of $C_t$ are as follows: when a pair of Brownian motions of the same colour meets frequently enough such that their collision local time exceeds an exponential time with rate $\gamma$, then one of the particles (chosen randomly) changes colour.

In order to describe the duality, we also introduce $L_t^=$, respectively $L_t^{\neq}$, as the total collision time collected up to time $t$ by all pairs of equal colour, respectively different colours.
Using the duality function 
\[  (u,v)^\MOJBssup{c}({\mathbf x}) := \prod_{i=1}^n (u,v)^\MOJBssup{c_i}(x_i) , \]
we can write the \emph{moment duality}\index{duality!moment duality} for the symbiotic branching model with finite $\gamma$ 
for $(u,v) \in (\mathcal{B}_{\rm tem}^+)^2$ and $c \in \{1,2\}^n$
as
\begin{equation}\label{MOJB-eq:moment_duality} \EE_{u,v} [ (u_t,v_t)^c({\mathbf x}) ] 
= \EE_{{\mathbf x}, c}\Big[ (u,v)^{C_t}(X_t) \,e^{\gamma (L_t^= + \varrho L_t^{\neq})}  \Big] , \end{equation}
see~\cite[Prop.\ 12]{MOJB-EF04}.

\begin{remark}{\em Critical curve.}
The duality holds for all $\varrho \in [-1,1]$, however only for $\varrho =-1$  it has been used to 
ensure uniqueness of solutions, see also the discussion at the end of Section 1.1 in~\cite{MOJB-EF04}. Moreover, by combining it with the self-duality, 
~\cite[Thm.\ 2.5]{MOJB-BDE11} show that there is a critical curve (as a function of $\varrho$) which determines which moments of the solutions remain bounded in time. More precisely, it is shown that
\begin{equation}\label{MOJB-moment_behavior}
p<p(\varrho) \quad \Rightarrow \quad  \EE_{\one,\one}\big[u_t(x)^p\big] \mbox{ is bounded uniformly in all }t\ge 0,\, x\in\RR,
\end{equation}
where $p(\varrho)=\frac{\pi}{ \arccos(-\varrho)}$ and $\one$ is the uniform initial condition.
\end{remark}

In order to derive a moment duality in the $\gamma = \infty$ case, 
the main idea of~\cite{MOJB-HOV15} is to decouple the evolution of the Brownian motions and the colourings in the following sense: we first sample the Brownian motions and then conditionally on the Brownian motions we treat the remainder of the right hand side of~\eqref{MOJB-eq:moment_duality} as a measure on 
colourings. More precisely, given the Brownian motions $X$ and an initial condition $c \in \{1,2\}^n$, we define a (random) measure on colourings by setting 
\begin{equation}\label{MOJB-eq:defn_M} M_t^\MOJBsse{\gamma}(b) := \EE_{c} \Big[ e^{\gamma (L_t^= + \varrho L_t^{\neq})} \one_{C_t = b} \, \Big| \, X_{[0,t]} \Big] , \end{equation}
for any $b \in \{1,2\}^n$. 

With this notation the duality~\eqref{MOJB-eq:moment_duality} can be written as
\[ \EE_{u_0,v_0} [ (u_t,v_t)^c({\mathbf x}) ] 
= \EE_{{\mathbf x}, c}\bigg[ \sum_{b \in \{1,2\}^n} (u_0,v_0)^\MOJBssup{b}(X_t) M_t^\MOJBsse{\gamma}(b) \bigg] . \]
This re-write of the duality is helpful, because we can write down the evolution of the measure $M^\MOJBsse{\gamma}$ explicitly.

\begin{lemma} 
With $L_t^{i,j}$ denoting the collision local time between $X^i$ and $X^j$, we have that
$M_0^\MOJBsse{\gamma} = \delta_c$ and
\begin{equation}\label{MOJB-eq:ODE_X_new}
 dM^\MOJBsse{\gamma}_t(b) 
= \frac{\gamma\varrho}{2} \sum_{ i,j =1 }^n \one_{b_i\neq b_j} M^\MOJBsse{\gamma}_t(b)\,dL_t^{i,j} + \frac{\gamma}{2}
 \sum_{i,j=1}^n \one_{b_i\neq b_j} M^\MOJBsse{\gamma}_t(\widehat{b}^i)\,dL_t^{i,j},
\end{equation}
where $\widehat{b}^i$ is the colouring $b$ flipped at $i$.
\end{lemma}

The lemma follows by looking at a small time increment and considering the possible changes in $M^\MOJBsse{\gamma}$ induced by the dynamics of the colouring process. The details can be found in~\cite[Lemma 3.4]{MOJB-HOV15}.

Notice that, conditional on $X$, \eqref{MOJB-eq:ODE_X_new} is  a system of linear ODEs driven by the  local times and also that increasing 
the parameter $\gamma$ corresponds to speeding up the time evolution.  Moreover, we have a lot of explicit control over the evolution of these ODEs in terms of eigenvalues and eigenvectors.
To be more explicit, suppose that the Brownian motions $X$ start in distinct starting positions. We set  
$\tau_0  = 0$  and
for $k \geq 0$ define
\[ \tau_{k+1} := \inf\{ t \geq \tau_k \, : \, \exists 
i \neq  j \, : \, X_t^i = X_t^j \mbox{ but } X_{\tau_k}^i \neq X_{\tau_k}^j, \} , \]
as the consecutive times when a new pair of Brownian motions meet. Then, we have $\tau_{k+1} > \tau_k$ almost surely and we can show that in the limit $\gamma \rightarrow \infty$,  in between times $\tau_k$ and $\tau_{k+1}$, the measures $M_t^\MOJBsse{\gamma}$ immediately settle into their equilibrium measure $M_t^\MOJBsse{\infty}$ (so are constant in the limit). We refer to~\cite[Theorems 2.3 and 2.5]{MOJB-HOV15}, where we show the following theorem:

\begin{theorem}\label{MOJB-theorem:M_t-infty}
Suppose that $\varrho \in [-1,- \cos(\pi/n)\wedge 0)$ and also that the starting point ${\mathbf x} = (x_1,\ldots, x_n) \in \RR^n$ satisfies $x_i \neq x_j$ for $i \neq j$. 
Then as $\gamma \rightarrow \infty$, the process $(M_t^\MOJBsse\gamma)_{t \geq 0}$ defined in~\eqref{MOJB-eq:defn_M} converges almost surely pointwise to a c\`agl\`ad\footnote{i.e.\ left-continuous with right limits}  limiting process $(M_t^\MOJBsse{\infty})_{t\ge0}$. 
Moreover, denote by $(u_t,v_t)_{t \geq 0}$
the (density of  the) infinite rate limit $\mathrm{SBM}(\varrho,\infty)_{u,v}$ started in $(u,v) \in (\cB_{\rm tem}^+)^2$. Then, for any colouring $c=(c_1,\ldots,c_n)\in\{1,2\}^n$ 
\begin{equation*}\label{MOJB-eq:moment duality infinite gamma}
\EE_{u,v} \Big[(u_t,v_t)^\MOJBssup{c}({\mathbf x})\Big]  = \EE_{{\mathbf x}}\Big[\sum_{b\in\{1,2\}^n}M_t^\MOJBsse{\infty}(b)\,(u,v)^\MOJBssup{b}(X_t)\Big], 
\end{equation*}
where $M_0^\MOJBsse{\infty}  = \delta_c$
and both sides are finite. 
\end{theorem}

The limiting process $M^\MOJBsse{\infty}$ is explicit, but fairly complicated to write down. 
Therefore, we refer to Propositions 4.2 and 4.6 in~\cite{MOJB-HOV15} for the details.

Note  for $\varrho =-1$ we have that $u_t +v_t$ satisfies the heat equation, i.e.\
$w_t:= u_t + v_t = S_t (u_0 + v_0)$, where $(S_t)_{t \geq0}$ denotes the heat semigroup.
In this case, the duality also simplifies  and we have the following result.

\begin{proposition}\label{MOJB-prop:duality_minusone}
Let $(u_t,v_t)_{t \geq 0}$ be the densities of cSBM$(-1,\infty)_{u,v}$ for initial conditions $(u,v) \in (\cB_{\rm tem}^+)^2$.
\begin{itemize}
\item[(i)] We have for ${\mathbf x} \in \RR^n$ with $x_1 < x_2 < \ldots< x_n$ and $c \in \{1,2\}^n$ an alternating colouring,
\[ \EE_{u,v} [ (u_t,v_t)^\MOJBssup{c}({\mathbf x})
] 
= \EE_{\mathbf x} [ (u,v)^\MOJBssup{c} (X_t) \one_{\{ t \leq \tau  \}}], \]
where $\tau := \inf\{ t \geq 0 \, : \, \exists i \neq j \, : \, X_t^i = X_t^j\}$ is the first collision time.
\item[(ii)] Suppose additionally that $u +v \equiv 1$. For any ${\mathbf x} = (x_1, \ldots,x_n) \in \RR^n$,
\begin{equation}\label{MOJB-eq:cts_voter_duality} \EE_{u,v}\Big[ \prod_{i=1}^n u_t(x_i) \Big]  = \EE_{\mathbf x} \Big[ \prod_{y \in {\mathbf Y}^{\mathbf x}_t} u(y) \Big], \end{equation}
where ${\mathbf Y}^{\mathbf x}_t=\{Y_t^\MOJBssup{x_i} \,:\, 1\leq i \leq n\}, t\ge0$, is a system of coalescing Brownian motions started from ${\mathbf x}$.
\end{itemize}
\end{proposition}

\begin{proof}[Proof sketch] In the case $\varrho =-1$, the dynamics of $M^\MOJBsse{\infty}$ 
can be described as follows, see~\cite[Prop.\ 4.2 and 4.6]{MOJB-HOV15}: 
Let $\tau_0 = 0$ and $(\tau_k )_{k \geq 1}$ be as before the consecutive time when a new pair of Brownian motions meets. Then, the process $M_t^\MOJBsse{\infty}$ is constant on each of the intervals $[0,\tau_1], (\tau_1,\tau_2], (\tau_2,\tau_3],\ldots$. 
Moreover, $M_t^\MOJBsse{\infty} = M_0^\MOJBsse{\infty} = \delta_c$ for $t \leq \tau_1$.

To describe the dynamics of  $M_t^\MOJBsse{\infty}$ for $t > \tau_1$,
define for $\ell_1,\ell_2 \in \{1, \ldots, n\}$ and any $b \in \{1,2\}^n$,
\begin{equation}\label{MOJB-eq:K_infty} K_\infty^{\ell_1,\ell_2}(\delta_b) = \left\{ \begin{array}{ll} 0 & \mbox{if } b_{\ell_1} \neq b_{\ell_2} ,\\
\delta_b + \frac 12 \delta_{\widehat b^{\ell_1}} + \frac12 \delta_{\widehat b^{\ell_2}} 
& \mbox{if } b_{\ell_1} = b_{\ell_2}, 
\end{array} \right. \end{equation}
where $\widehat b^{\ell}$ denotes the colouring flipped at position $\ell$.
Then define $K_\infty^{\ell_1,\ell_2}(M)$ for any measure $M$ on $\{1,2\}^n$ by setting 
\[ K_\infty^{\ell_1,\ell_2}(M) = \sum_{b \in \{1,2\}^n} M(b)\,  K_\infty^{\ell_1,\ell_2}(\delta_b) . \]

Now, we can define $M_t^\MOJBsse{\infty}$ inductively. 
Assume that at time $\tau_k$ with $k \geq 1$ the pair of Brownian motions with indices $\ell_1$ and $\ell_2$ meets, then set
\[ M_t^\MOJBsse{\infty} = K_\infty^{\ell_1,\ell_2}(M_{\tau_k}^\MOJBsse{\infty}) \quad \mbox{for } 
t \in (\tau_k, \tau_{k+1}] . \]

(i) In the first case, we see that $M_0^\MOJBsse{\infty} = \delta_c$ for $t \leq \tau = \tau_1$.
In particular, if at time $\tau$ the Brownian motions indexed by $\ell_1,\ell_2$ meet, we have by assumption on $c$ that $c_{\ell_1} \neq c_{\ell_2}$. Hence, $K_\infty^{\ell_1,\ell_2} = 0$ and 
thus $M_t^\MOJBsse{\infty} = 0$ for all $t > \tau = \tau_1$. The statement of (i) then follows from Theorem~\ref{MOJB-theorem:M_t-infty},  also see~\cite[Lemma 5.8]{MOJB-HOV15} for a more detailed argument.

(ii) We only prove the case when $n =2$. The general case follows along similar lines, but the combinatorics is more involved (also see Remark~\ref{MOJB-rem_alternative} for an alternative approach).
We have by Theorem~\ref{MOJB-theorem:M_t-infty} that for $c = (1,1)$, 
\[ \EE_{u,v}\Big[  u_t(x_1) u_t(x_2)\Big]   =
\EE_{{\mathbf x}}\Big[\sum_{b\in\{1,2\}^2}
 M_t^\MOJBsse{\infty}(b)\,(u,v)^\MOJBssup{b}(X_t)\Big],
\]
where $M_0^\MOJBsse{\infty} = \delta_c$.
As before, we have $M_t^\MOJBsse{\infty}  = \delta_c = \delta_{(1,1)}$ for $t \in [0,\tau_1]$. 
Then, for $t > \tau_1$, we have that
\[    M_t^\MOJBsse{\infty}  = K_\infty^{\ell_1,\ell_2}(\delta_{(1,1)})
= \delta_{(1,1)} + \frac 12 \delta_{(1,2)} + \frac 12 \delta_{(2,1)}. \]
Thus, using that $u + v = 1$, we get that
\[ \begin{aligned} \EE_{u,v}& \Big[   u_t(x_1) u_t(x_2)\Big]  
\\
&  = \EE_{{\mathbf x}}\big[ \one_{\{ t \leq \tau_1\}} u(X_t^1)u(X_t^2) \big]
+ \EE_{{\mathbf x}}\big[ \one_{\{t > \tau_1\}} u(X_t^1)u(X_t^2) \big]
\\ & \quad  + \frac 12\EE_{{\mathbf x}}\big[ \one_{\{ t > \tau_1\}} u(X_t^1)(1-u)(X_t^2) \big]
+ \frac 12\EE_{{\mathbf x}}\big[ \one_{\{ t > \tau_1 \}} (1-u)(X_t^1)u(X_t^2) \big] \\
& = \EE_{{\mathbf x}}\big[ \one_{\{t \leq \tau_1\}} u(X_t^1)u(X_t^2) \big]
+ \frac 12 \EE_{{\mathbf x}}\big[ \one_{\{ t > \tau_1\}} u(X_t^1) \big]
+ \frac 12 \EE_{{\mathbf x}}\big[ \one_{\{ t > \tau_1 \}} u(X_t^2) \big].
\end{aligned} \]
This  expression can be represented in terms of coalescing Brownian motions as in (ii), if we notice that we can obtain a system of coalescing Brownian motions from $X$ by deciding at the collision time to follow exactly one of the two Brownian motions, chosen each with probability $1/2$.
\end{proof}

\begin{remark}\label{MOJB-rem_alternative}
An alternative derivation of the second part of the proposition would be to recall that the case $\varrho =-1$ and $u + v \equiv 1$ corresponds to the infinite rate limit of the continuous-space stepping stone model. 
\cite[Thm.\ 4.1]{MOJB-Shiga86} showed that in the finite $\gamma$-case there is a similar moment duality, where however the dual is a system of delayed coalescing Brownian motions,\index{coalescing Brownian motions} where two motions only coalesce at rate $\gamma \times$ their collision local time. Taking $\gamma \rightarrow \infty$ in the dual gives the instantaneously coalescing Brownian motions as in part (ii).
\end{remark}

\section{Interface duality in the symbiotic branching model}\label{MOJB-sec:interface_duality}

In the case $\varrho =-1$, duality allows us to 
explicitly characterize the interface process in the symbiotic branching model, see also Section~\ref{MOJB-sec:discrete_voter} for a similar result in the discrete voter model.
The following results generalize~\cite{MOJB-T95} to general initial conditions that have infinitely many interface points and we can also remove the restriction that the initial conditions  satisfy $u + v \equiv 1$.

Define  $\cU$ as the space of absolutely continuous measures $(u,v)$ with bounded densities also denoted by $(u,v)$ such that $u(x) v(x) = 0$ and $u(x) + v(x) > 0$ for almost all $x \in \RR$.
 For $(u,v) \in\cU$, we define
\begin{align}\label{MOJB-def:interface}
\mathcal{I}(u,v) := \supp(u)\cap \supp(v),
\end{align} 
where $\supp(u)$ denotes the measure-theoretic support.

Our next result deals with initial conditions of `single interface point' type:

\begin{theorem}\label{MOJB-thm:annihilating-BM-_onepoint}
Assume $(u,v) \in \cU$ such that $|\mathcal{I}(u,v)| = 1$.
Let $(u_t,v_t)_{t\ge0}$ denote the solution of
$\mathrm{cSBM}(-1,\infty)_{u,v}$.
Then we have, almost surely, $|\mathcal{I}(u_t,v_t)| =1$ for all $t \geq 0$ and if we denote 
by $I_t$  the single interface point, then 
  almost surely $(I_t)_{t \geq 0}$ is continuous and
  there exists a standard Brownian motion $(B_t)_{t \geq 0}$ 
 such that $I_t$ is the unique (in law) weak solution\footnote{Under the assumption $(u,v)\in\cU_1$, the integrand on the right hand side of \eqref{MOJB-eq:SDE} is not guaranteed to be Lebesgue-integrable at $0$. However, in any case the integral exists as an improper integral $\lim_{\varepsilon\downarrow0}\int_\varepsilon^t \frac{w'_s(I_s)}{w_s(I_s)}\, ds$.}
 of
\begin{align}\label{MOJB-eq:SDE}
I_t = I_0 - \int_0^t \frac{w'_s(I_s)}{w_s(I_s)}\, ds +B_t, \quad t \geq 0, 
\end{align}
where $w_t = S_t (u+v)$, for $(S_t)_{t \geq 0}$ the heat semigroup.
Moreover, if $\one_{\{ u(x) > 0 \}} \rightarrow 1$ as $x \rightarrow -\infty$, then 
\[ u_t(x) = \one_{\{ x \leq I_t\}} w_t(x) \quad \mbox{and}\quad
v_t(x) = \one_{\{ x \geq I_t \}} w_t(x) . \]
and otherwise the roles of $u_t$ and $v_t$ have to be interchanged.
\end{theorem}

\begin{proof}[Proof sketch]
 Using the duality relation for the mixed moments in Proposition~\ref{MOJB-prop:duality_minusone} (i), we can show that the number of interfaces cannot increase. The one-dimensional distributions for a single $I_t$ follow from  a first moment calculation using that since $u_t(x) \in \{0,w_t(x)\}$, we have that $\PP(I_t \geq x ) = \EE[ \frac{u_t(x)}{w_t(x)}]$ and then applying duality. 
\end{proof}

Suppose now that the initial condition $(u,v)\in \cU$ is such that $\mathcal{I}(u,v)$ has no accumulation points. Define $m(u,v,x)$ to be $1$
if  $x \in \supp(u)$ and set it to be $2$ otherwise. 
Suppose $(Y_t^x)_{t \geq 0}$, $x \in \mathcal{I}(u,v)$ is a system of stochastic processes with $Y_0^x = x$ 
that move independently according to the stochastic differential equation~\eqref{MOJB-eq:SDE} until 
two motions collide at which point the pair annihilates. Then, the paths of the annihilating motions induce a partition of the set $[0,\infty)\times \RR$. We can define a `colouring' of the half-plane
by defining $\hat m(t,x)$ such that  $\hat m(0,x) = m(u,v,x)$ and
such that the each component of the partition has the same colour $\in \{1,2\}$ 
(and such that the boundaries are of colour $1$), see Figure~\ref{MOJB-fig:colouring-1} for an illustration. Then, we can define, with $w_t = S_t(u+v)$,
\[ \hat u_t(x) = w_t(x) \one_{\{ \hat m(t,x) = 1 \}}, \quad \mbox{and}\quad \hat v_t(x) = w_t(x) \one_{\{ \hat m(t,x) = 2\} } . \]

\begin{figure}[t]
\centering
\includegraphics[width=5cm]{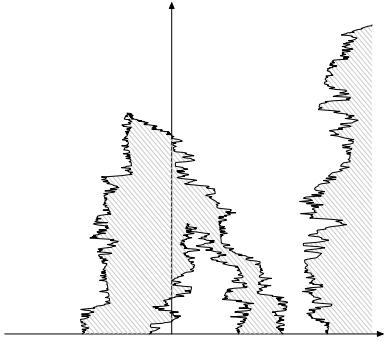}
\caption{\label{MOJB-fig:colouring-1}
An illustration of the colouring $\hat m$ of $[0,\infty)\times\RR$ induced by an initial configuration $(u_0,v_0)$ with five interfaces.
Type 1 is drawn in white and type 2 is shaded grey.
Figure taken from~\cite{MOJB-HOV17}.
}
\end{figure}

\begin{theorem}[{\cite[Thm.\ 2.12]{MOJB-HOV17}}]\label{MOJB-theorem:annihilating-BM}
Assume that $(u,v) \in \cU$ and $\mathcal{I}(u,v)$ has no accumulation points. 
Let $(u_t,v_t)_{t\ge0}$ denote the infinite rate limit\linebreak $\mathrm{cSBM}(-1,\infty)_{u,v}$.
Then, 
\[ (u_t,v_t)_{t \geq 0} \stackrel{d}{=} (\hat u_t,\hat v_t)_{t \geq 0}, \]
where $(\hat u_t,\hat v_t)_{t \geq 0}$ is defined in terms of annihilating motions with dynamics given by~\eqref{MOJB-eq:SDE} as above.
\end{theorem}

\begin{remark}
\begin{itemize}
\item[(a)] We also have  a version of this theorem for arbitrary initial conditions, see~\cite[Thm.\ 2.14]{MOJB-HOV15}. There we show that there is a `coming down from infinity'\index{coming down from infinity} effect and
for any $t > 0$, the set $\mathcal{I}(u_t,v_t)$ does not have any accumulation points and the evolution of $(u_s,v_s)_{s \geq t}$ is as above, but started from $(u_t,v_t)$. The notion of `coming down from infinity' originates in coalescent theory, see also the contributions of Blath/Kurt, Kersting/Wakolbinger and Birkner/Blath in this volume.
\item[(b)] In the case $u + v =1$, there is a further interesting duality to annihilating Brownian motions. For the finite $\gamma$ stepping stone model cSBM$(-1,\gamma)$, ~\cite[Lemma 2.1]{MOJB-BEM07} implies (in the special case $s=0$ and there formulated for discrete space) that
\[ \EE_{u} \Big[ \prod_{i=1}^n (1- 2 u_t(x_i)) \Big] = \EE_{\mathbf x}\Big[ \prod_{y \in {\mathbf X}_t} (1 - 2 u(y))\Big] , \]
{for }${\mathbf x} = (x_1, \ldots, x_n)$ and
where $({\mathbf X}_t)_{t \geq 0}$ is a system of delayed annihilating Brownian motions started in ${\mathbf x}$. Taking the limit $\gamma \rightarrow \infty$ on the left will then  lead to instantaneously annihilating Brownian motions on the right.
\end{itemize}
\end{remark}

Note in the case that $\varrho =-1$ and the initial conditions are such that $u+v \equiv 1$, the duality in Proposition~\ref{MOJB-prop:duality_minusone}(ii) is the continuous-space analogue of the discrete-space voter model duality in~\eqref{MOJB-eq:voter_duality}. 
Indeed, the result shows that  cSBM$(-1,\infty)$ corresponds to the \emph{continuous-space voter model}\index{voter model!continuous-space} first introduced by~\cite{MOJB-Evans1997}
and further discussed in~\cite{MOJB-Donnellyetal2000} and~\cite{MOJB-Zhou2003}, where it is referred to as \emph{continuum-sites stepping-stone model}.

In analogy with the discrete model from Section~\ref{MOJB-sec:discrete_voter}, 
the continuous-space voter model can also be obtained 
via a graphical construction. We will give an informal description, the idea goes back to~\cite{MOJB-AS11}, see also~\cite{MOJB-GSW16} for further details in a more general situation.

\index{graphical construction}
\index{Brownian web}

Let $\cW = (\cW^\MOJBssup{t,x})_{(t, x) \in \RR^2}$ be the (time-reversed) Brownian web,
informally this is a system of coalescing Brownian motions $\cW^\MOJBssup{t,x}$ started at every point $(t,x)\in\RR^2$ with $\cW^\MOJBssup{t,x}_0 = x$, see~\cite{MOJB-SSS17} for technical details. 
Recall that since $u_t = 1- v_t$, we only need to construct $u_t$. Fix initial conditions 
$u$ such that $u(x) \in [0,1]$ for every $x \in \RR$.

In order to determine the state of the system $u_t(x)$ at some time  $t\geq 0$ and location $x \in \RR$, one
traces back the genealogy by following $\cW^\MOJBssup{t,x}$ for $t$ time units, then samples a type  $\chi_{\cW^\MOJBssup{t,x}_0}$ in $\{0,1\}$ according to a Bernoulli variable with success probability $u(\cW^\MOJBssup{t,x}_t)$. Finally set $u_t(x) = \chi_{\cW^\MOJBssup{t,x}_t}$.
 Note that this  gives a pathwise construction of the continuous-space voter model which is completely analogous to the discrete construction in Section~\ref{MOJB-sec:discrete_voter}. For a more careful construction which takes care of the non-trivial difficulties we refer to~\cite{MOJB-GSW16}.

In particular, one can easily deduce the moment duality~\eqref{MOJB-eq:cts_voter_duality}, since for ${\mathbf x} = (x_1, \ldots,x_n)$,
\[ 
\prod_{i=1}^n u_t(x_i) = \prod_{i=1}^n \chi_{\cW^\MOJBssup{t,x_i}_t} . \]
Then, taking expectations and carrying out the expectation over 
the Bernoulli variables, gives
\[ \EE_u \Big[ \prod_{i=1}^n u_t(x_i) \Big] = \EE \Big[ \prod_{i=1}^n u(\cW^\MOJBssup{t,x_i}_t) \Big] , \]
and the system $\cW^\MOJBssup{t,x_i}_t, i =1,\ldots,n$ is in law equivalent to the system of coalescing Brownian motions started in ${\mathbf x}$.

As in the discrete case,
the same construction can also be used to see that the interfaces in the continuous-space voter model are given by a system of annihilating Brownian motions, and thus giving an alternative proof of Theorem~\ref{MOJB-theorem:annihilating-BM},
see~\cite{MOJB-HOV17} for details.

\section{Entrance laws for annihilating Brownian motions}\label{MOJB-sec:entrance_laws}

In the last section we have seen that we can use the moment duality for the symbiotic branching model with $\varrho = -1$ to show that the interfaces between different types behave
like (instantaneously) annihilating Brownian motions (aBMs).\index{annihilating Brownian motions} It turns out that this observation can be used to
gain insight into the behaviour of aBMs itself. This section is based on \cite{MOJB-HOV17}.

The following is a well-known issue that arises when trying to construct aBMs: 
As long as the set of starting points is finite, it is absolutely straightforward to construct 
a system of aBMs. Even if the initial points form a locally finite (equivalently discrete and closed) subset of the real  line, then 
it is possible to construct the corresponding systems, even though a little care is needed, see
 \cite[Sec.\ 4.1]{MOJB-TZ11} or~\cite[Appendix A.1]{MOJB-HOV17}.
 However, if one considers a sequence of locally finite starting points that become dense in the real line, it is not  at all clear if the corresponding system of aBMs converges in some suitable sense.

Mathematically, this question is closely related to characterizing entrance laws for aBMs. 
Recall that a family $\mu=(\mu_t)_{t>0}$ of probability measures on (the Borel $\sigma$-algebra of)  a suitable state space $\cD$ is called a \emph{probability entrance law}\index{entrance law} for a semigroup $(P_t)_{t\ge0}$ if 
\begin{equation}\label{MOJB-eq:entrance_law}\mu_sP_{t-s}=\mu_t \qquad \text{for all }0<s<t.\end{equation} 
See e.g. \cite[ Appendix A.5]{MOJB-Li} or \cite{MOJB-Sharpe} for the general theory of entrance laws.
Roughly speaking, an entrance law corresponds to a Markov process $({\mathbf X}_t)_{t>0}$ with time-parameter set $(0,\infty)$, whose one-dimensional distributions are given by~$\mu_t$, but where we do not specify the initial condition.

\index{entrance law}
\index{annihilating Brownian motions}
\index{coalescing Brownian motions}

One approach to finding entrance laws for aBMs, carried out in \cite{MOJB-TZ11}, is to use a thinning relation between annihilating and coalescing Brownian motions. 
Unlike for aBMs, one can always add coalescing Brownian motions to an existing system in a consistent way. In particular, if the starting points become dense everywhere on the real line, this leads to a unique entrance law for cBMs, also known as Arratia's flow~\cite{MOJB-A79}.
Using the thinning relation, see~\cite[Sec 2.1]{MOJB-TZ11} this leads also to an entrance law for aBMs (termed `maximal' in~\cite{MOJB-TZ11}).

However, it is also clear that the consistency of the cBMs described above does not hold for aBMs  and
different ways of taking initial conditions that become denser may lead to different system of aBMs.
Our main result in~\cite{MOJB-HOV17} is to classify all possible entrance laws for systems of aBMs
using the connection to cSBM$(-1,\infty)$ described in Section~\ref{MOJB-sec:interface_duality}.
Before we can state this correspondence, we need to set up a bit of notation.

We expect that even if we start with a dense set of points on the line, then at any positive $t>0$
many nearby Brownian motions would have annihilated and so  the positions of the remaining Brownian motions would form a discrete (and closed) subset of $\RR$.
Hence, a suitable state space for the evolution of aBMs is given by
\[\cD:=\{{\mathbf x}\subseteq\RR:{\mathbf x}\text{ is discrete and closed}\}.\]
For each ${\mathbf x}\in\cD$ a system of aBMs starting from ${\mathbf x}$ can be constructed as a (strong) Markov process ${\mathbf X}^{\mathbf x}=({\mathbf X}^{\mathbf x}_t)_{t\ge0}$ taking values in $\cD$.

The main idea is to use the fact that any system of aBMs started in $\cD$ corresponds to 
the interfaces in the symbiotic branching model with $\varrho = -1$ and the right initial conditions.
So let
\[ \cM_1:=\{u(x)\,dx\,|\,u:\RR\to[0,1]\text{ measurable}\}\]
denote the space of all absolutely continuous measures on $\RR$ with densities taking values in $[0,1]$. 
Here, we recall that since for $\varrho =-1$ and initial conditions $u + v \equiv 1$,  in order to specify cSBM$(-1,\infty)_{u,v}$ we only need to describe the evolution of one type
and so choosing $u \in \cM_1$ fixes the initial conditions.

To go from the measure-valued process to the aBMs, we need to define a mapping that associates to each $u \in \cM_1$ its interface. So we think of $u(x)$ representing the proportion of type $1$ particles at $x$ and correspondingly $1-u(x)$ as the proportion of type $2$ particles. Thus, in agreement with~\eqref{MOJB-def:interface} we define
for $u\in\cM_1$, 
\begin{equation}\label{MOJB-eq:interface_u} \mathcal{I}(u):=\mathcal{I}(u,1-u) = \supp(u)\cap\supp(1-u),\end{equation}
as the interface points between the two types,
where we recall that $\supp(u)$ denotes the measure-theoretic support of $u$.

By Theorem~\ref{MOJB-theorem:annihilating-BM}  any initial condition $u$ such that the interface
$\mathcal{I}(u) \in \cD$  corresponds to a system of aBMs started in the points $\mathcal{I}(u)$. 
Conversely, the interfaces only describe the population model up to interchanging the types.
Therefore, we define an equivalence relation $\sim$ on $\cM_1$
by identifying two measures with densities $u$ and $v$ iff 
either almost everywhere $u = v$ or almost everywhere  $u = 1-v$. Then, we 
 work with the quotient space
\[ \cV := \cM_1 /\! \sim.\]
We write $v=[u]=\{u,1-u\}$ for elements of $\cV$.
Now, using that on the level of measure-valued processes it is possible to start in any measure 
$u \in \cM_1$ we have the following result.

 \begin{theorem}[\cite{MOJB-HOV17}]\label{MOJB-thm:entrance_law}
 There is a suitable topology on $\cD$ such that there is a bijective correspondence between probability entrance laws $\mu=(\mu_t)_{t>0}$ for the semigroup $(P_t)_{t\ge0}$ of aBMs on $\cD$ and probability measures $\nu$ on $\cV$.
 \end{theorem}
 
 Informally, the measure $\nu$ (up to fixing types) will correspond to the law of the initial condition of the measure-valued process, whose interfaces then correspond to the system of aBMs.
 The right choice of topology is absolutely essential for the above theorem to work. 
 The obvious choice of identifying points in $\cD$ with locally finite point measures and then relying on the vague convergence is not a good choice. One reason for that is that in this topology we only get c\`adl\`ag, but not continuous, paths for the process ${\mathbf X}$.

Instead, we construct a weaker topology on $\cD$ which is derived from the topology on the level of the measure-valued process.
We start by equipping $\cM_1$ with the vague topology, this turns
$\cM_1$ into compact space, see e.g.~\cite[Appendix~A.2]{MOJB-HOV17}.
An important role is played by the 
the subspace of all $u\in\cM_1$ with discrete interface 
\[\cM_1^d:=\{u\in\cM_1\,|\,\mathcal{I}(u)\in\cD\},\]
which is dense in $\cM_1$, see~\cite[Appendix A.2]{MOJB-HOV17}.
Note that for each $u\in\cM_1^d$, we may choose a version of its density
that is locally constant on the complement of $\mathcal{I}(u)$.
Again we need to take into consideration that $u$ and $1-u$ give rise to the same interface and
so set
\[ \cV^d := \cM_1^d /\! \sim, \]
where $\sim$ is as before. On this quotient space the mapping $\mathcal{I}$ defined in
~\eqref{MOJB-eq:interface_u} induces a mapping $\mathcal{I} : \cV^d \rightarrow \cD$
which is well-defined and bijective.
In particular, it induces  a topology on $\cD$, generated by the system
$\{\mathcal{I}(U): U\subseteq\cV^d\text{ open}\}$,
which is by definition  the coarsest topology on $\cD$ with respect to which $\mathcal{I}^{-1}:\cD\to\cV^d$ is continuous.  
This topology is the one referred to in Theorem~\ref{MOJB-thm:entrance_law}.

Our next theorem makes the connection to the infinite-rate symbiotic branching model more explicit and also clarifies what happens when taking a sequence of initial conditions for aBMs that become dense on the real line.

\begin{theorem}[\cite{MOJB-HOV17}]\label{MOJB-thm:convergence}
Let $\cD$ be endowed with the above topology and
let $(\mu^\MOJBssup{n})_{n\in\NN}$ be a sequence of probability measures on $\cD$. Consider the corresponding sequence 
of aBM processes $({\mathbf X}^\MOJBssup{n}_t)_{t\geq 0}$ started according to the (random) initial condition $\mu^\MOJBssup{n}$. 
Then $({\mathbf X}^\MOJBssup{n}_t)_{t>0}$ converges in distribution in $\cC_{(0,\infty)}(\cD)$ 
iff the sequence $(\mu^\MOJBssup{n}\circ\mathcal{I})_{n\in\NN}$ of probability measures on $\cV^d$ converges weakly to some probability measure $\nu$ on $\cV$, in which case as $n \rightarrow \infty$
\begin{equation}\label{MOJB-eq:convergence}
 ({\mathbf X}^\MOJBssup{n}_t)_{t>0} \stackrel{d}{\longrightarrow} (\mathcal{I}(V_t))_{t>0}\quad\text{on }\cC_{(0,\infty)}(\cD).
\end{equation}
Here, $V_t = [u_t] \in \cV$ for $(u_t, 1-u_t)_{t \geq 0}$  a solution of cSBM$(-1,\infty)_{u,1-u}$ and where the (random) initial conditions $u$ are 
chosen by first choosing $V \in \cV$ according to $\nu$ and then taking $u$
such that $[u] = V$.
Moreover, $\mathcal{I}(V_t) = \mathcal{I}(u_t)$ as defined in~\eqref{MOJB-eq:interface_u}.
\end{theorem}

We will illustrate the power of the correspondence between aBMs and the population model by looking at the following examples.

\begin{example}
\begin{itemize}
 \item[(i)] If we take ${\mathbf x}_n=\frac1n\ZZ$, then $\mathcal{I}^{-1}({\mathbf x}_n)$ converges to $[\frac12]$ in $\cV$, and hence by Theorem~\ref{MOJB-thm:convergence} the system of aBMs starting from ${\mathbf x}_n$ converges. 
\item[(ii)] Suppose that ${\mathbf x}_n$ is distributed according to a Poisson point process on $\RR$ with intensity $n$, then $\mathcal{I}^{-1}({\mathbf x}_n)$ also converges to $[\frac12]$.
\item[(iii)] Now, consider ${\mathbf x}_n=\frac1n\ZZ+\{0,\frac1{n^2}\}$ 
then $\mathcal{I}^{-1}({\mathbf x}_n)$ still converges in $\cV$, however now the limit is $[0]$, so the limiting system of aBMs corresponds to the empty system. Intuitively, this corresponds to the case where each pair of nearby aBMs starts so close to each other that cancellations arise so early that they do not survive in the limit.
\item[(iv)] Finally, we  consider ${\mathbf x}_n=\frac1n\ZZ+\{0,\frac1{4n}\}$, where
now  $\mathcal{I}^{-1}({\mathbf x}_n)$ converges to $[\frac14]$ in $\cV$, which is different from $[\frac12]$. 
So the examples in (i), (ii) and (iv) converge, but 
in the latter example the aBMs `come down from infinity' in a different way, thus leading to a different entrance law. 
\end{itemize}

\index{coming down from infinity}

Note  that the approximations in (i) and (ii) give the `maximal' entrance law considered in \cite{MOJB-TZ11}. See Figure~\ref{MOJB-fig:simulations} for illustrations of the four examples. Finally, we also observe that it is possible to construct initial conditions in $\cD$ such that $\mathcal{I}^{-1}({\mathbf x}_n)$ does not converge to 
any measure, e.g.\ by starting with finitely many points and then adding one point at a time such that an accumulation point arises. The corresponding system of aBMs starts either with an odd or even number of motions, so cannot converge as in either case the system will die out while in the other case one motion survives. See~\cite{MOJB-HOV17} for more details.
\end{example}

\begin{figure}[t]
\begin{center}
      \includegraphics[width=0.7\textwidth]{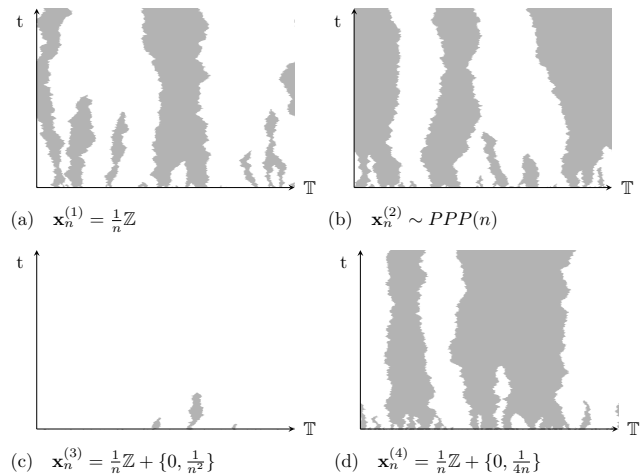}
\end{center}
\caption{\label{MOJB-fig:simulations} Simulations of aBMs {as interface process} with discrete starting configurations {on a torus $\TT$}, taken from~\cite{MOJB-HOV17}.}
\end{figure}

Finally, the construction also gives insight into $n$-point densities for aBMs, which 
for an entrance law $\mu=(\mu_t)_{t>0}$ are defined as
\begin{align}\label{MOJB-eq:density} 
p_{\mu}(t,{\mathbf x})&:=
\lim_{\varepsilon\to0}\frac{1}{(2\epsilon)^n}\,\PP_{\mu}\Big(\bigcap_{i=1}^n\{{\mathbf X}_t\cap[x_i-\varepsilon,x_i+\varepsilon]\ne\emptyset\}\Big),
\end{align}
for ${\mathbf x} = (x_1, \ldots, x_n)$ with $x_1< \ldots< x_n$ and $t>0$.
The relation to cSBM$(-1,\infty)$ (or the continuous-space voter model) and the construction in terms of the Brownian web discussed above gives us a way to calculate the one-point densities explicitly as well as a to simplify the expression for general $n$-point densities. See~\cite{MOJB-HOV17} for details.

\section{Outlook}\label{MOJB-sec:outlook}

We have seen that the symbiotic branching model has a very rich duality theory which allows us to carry out a fairly detailed analysis of this system of SPDEs. 
One of the main open problems is to characterize the $\gamma = \infty$ limit more explicitly for $\varrho \in (-1,0)$ and possibly to extend the results to non-negative $\varrho$.

Another challenging direction would be to consider models with large range migration, where the Laplace term in~\eqref{MOJB-eqn:spde} is replaced by e.g.\ a fractional Laplacian. While some of the duality theory still works, it is unclear what the scaling limit would be even in the stepping stone case $\varrho = -1$.

Finally, it would be very interesting to see how far our techniques of characterising the limiting process without specifying the correlation has applications in other spatial systems where types naturally separate, such as some of the limiting objects found in~\cite{MOJB-BEV13} or the two-dimensional version of the mutually catalytic branching model~\cite{MOJB-Dawsonetal2002a}.

{\bf Acknowledgments.} We are indebted to our collaborators 
Matthias Hammer and Florian V\"ollering.
%for collaborating on this project. 
We also would like to thank the two referees and Matthias Hammer for carefully reading our manuscript and for their helpful comments.

\end{document}